\documentclass[11pt]{amsart}

\usepackage{amsmath, amsthm, amssymb, amsfonts, verbatim}
\usepackage[pdftex]{graphicx}
\usepackage[bookmarks]{hyperref}

\renewcommand{\baselinestretch}{1}
\reversemarginpar            

\def\be{\begin{equation}}
\def\ee{\end{equation}}
\def\qand{\quad\mbox{and}\quad}

\def\H{{\mathcal H^{s,\theta}}}
\def\Hs{{\mathcal H^{s+1,\theta}}}
\def\h{{H^{s,\theta}}}

\def\B{{\mathcal B}}

\def\p{{\tilde p}}

\DeclareMathOperator{\diag}{diag}

\def\R{{\mathbb R}}
 
\def\bs{\begin{split}}

\theoremstyle{plain}
\newtheorem*{main}{Main~Theorem}

\def\lm{\begin{lem}}
\def\ml{\end{lem}}
\newtheorem{thm}{Theorem}[section]
\newtheorem{cor}[thm]{Corollary}
\newtheorem{lem}[thm]{Lemma}
\newtheorem{prop}[thm]{Proposition}

\theoremstyle{definition}

\theoremstyle{remark}
\newtheorem{remark}{Remark}[section]

\numberwithin{equation}{section}

\def\be{\begin{equation}}
\def\ee{\end{equation}}
\def\qand{\quad\mbox{and}\quad}

\def\R{\mathbb R}

\def\pru{\begin{proof}}
\def\urp{\end{proof}}

\theoremstyle{definition}

\newtheorem*{theorem*}{Theorem}
\newtheorem*{corollary*}{Corollary}
\newtheorem*{acknowledgment}{Acknowledgments}

\begin{document}

\title[Wellposedness for the $2+1$ Monopole Equation]{Local Wellposedness\\ for the $2+1$ Dimensional Monopole Equation}


\author{Magdalena Czubak}
\email{czubak@math.toronto.edu}
\address{Department of Mathematics, 
         University of Toronto}
\begin{abstract}
The space-time monopole equation on $\R^{2+1}$ can be derived by a dimensional reduction of the anti-self-dual Yang Mills equations on $\R^{2+2}$.  It can be also viewed as the hyperbolic analog of Bogomolny equations.  We uncover null forms in the nonlinearities and employ optimal bilinear estimates in the framework of Wave-Sobolev spaces.  As a result, we show the equation is locally wellposed in the Coulomb gauge for initial data sufficiently small in $H^s$ for $s>\frac{1}{4}$. 
\end{abstract}
\subjclass[2000]{70S15, 35L70;}
\keywords{monopole, null form, Coulomb gauge, wellposedness;}

\maketitle
\begin{section}{Introduction}\label{intro}
In this paper we study local wellposedness of the Cauchy problem for the monopole equation on $\R^{2+1}$ Minkowski space in the Coulomb gauge.  The space-time monopole equation can be derived by a dimensional reduction from the anti-self-dual Yang Mills equations on $\R^{2+2}$, and is given by 
\be\tag{ME}
F_A=\ast D_A\phi,
\ee
where $F_A$ is the curvature of a one-form connection $A$ on $\R^{2+1}$, $D_A\phi$ is a covariant derivative of the Higgs field $\phi$, and $\ast$ is the Hodge star operator with respect to the Minkowski $\R^{2+1}$ metric.  (ME) is a hyperbolic analog of Bogomolny equations, and was first introduced by Ward \cite{Ward89} and discussed from the point of view of twistors.  Ward also studied its soliton solutions \cite{Ward99}.  Recently, Dai, Terng and Uhlenbeck gave a broad survey on the space-time monopole equation in \cite{DCU}.  In particular, using the inverse scattering transform they have shown global existence and uniqueness up to a gauge transformation for small initial data in $W^{2,1}$.  However, $L^2$ based wellposedness theory for this equation has not been investigated.  The objective of this paper is to fill this gap by specifically treating the Cauchy problem for rough initial data in $H^s$.\\
\indent
Written in coordinates, (ME) is a system of first order hyperbolic partial differential equations.  The unknowns are a pair $(A,\phi)$.
If $(A,\phi)$ solve the equation, then so do 
\[
\lambda A(\lambda t, \lambda x) \qand \lambda \phi(\lambda t, \lambda x),
\]
for any $\lambda>0$.  This results in the critical exponent $s_c=0$.  Since in general one expects local wellposedness for $s>s_c$ the goal would be to show (ME) is wellposed for $s > 0$.  Nevertheless, the two dimensions create an obstacle, which so far only allows $s > \frac{1}{4}$.  We explain this now.  In Section \ref{mesys} we choose a Coulomb gauge, and reformulate (ME) as a system of semilinear wave equations coupled with an elliptic equation, to which we refer as auxiliary monopole equations (aME).  Schematically it looks as follows
\be\tag{aME}
\bs
\square u&=\mathcal B_+(\partial u,\partial v,A_0),\\
\square v&=\mathcal B_-(\partial u,\partial v,A_0),\\
\triangle A_0&=\mathcal C(\partial u, \partial v, A_0),
\end{split}
\ee
where $\mathcal B_\pm, \mathcal C$ are bilinear forms\footnote{See Section \ref{mesys} for the precise formula for  $\mathcal B_\pm$ and $\mathcal C$.}, $A_0$ is the temporal part of the connection $A$, $\partial u, \partial v$ denote space-time derivatives of $u$ and $v$ respectively, and are given in terms of $\phi$ and the spatial part of $A$.   As a result, showing wellposedness of (ME) for $s>0$ can follow from showing (aME) is wellposed for $s>1$ (see Theorem \ref{returnME}).  Also, the most difficult nonlinearity that we have to handle is contained in $\mathcal B_\pm(\partial u,\partial v,A_0)$. Luckily, it exhibits a structure of a null form.  There are two standard null forms:
\begin{align*}
Q_0(u,v)&=-\partial_tu \partial_tv + \nabla u \cdot\nabla v,\\
Q_{\alpha\beta}(u,v)&=\partial_\alpha u \partial_\beta v - \partial_\beta u \partial_\alpha v.
\end{align*}
The null condition was introduced by Klainerman \cite{Klainerman83}, and it was first applied to produce better local wellposedness results for wave equations with a null form by Klainerman and Machedon in \cite{KlainermanMachedon93}. Indeed, in low dimensions, for these kind of nonlinearities one can assume much less regularity of the initial data than for the general products.  Counterexamples for general products were shown by Lindblad \cite{Lindblad96}.  We uncover the null form $Q_{\alpha\beta}$ in our system of wave equations as well as a new type of a null form which is related to $Q_{\alpha\beta}$.  Unfortunately, the results in two spatial dimensions for $Q_{\alpha\beta}$ are not as optimal as they are in higher dimensions or as they are for $Q_0$.  In fact, the best result in literature so far for $Q_{\alpha\beta}$ in two dimensions is due to Zhou in \cite{Zhou}.  He establishes local wellposedness for initial data in $H^s \times H^{s-1}$ for $s > \frac{5}{4}$.  In addition, by examining the first iterate Zhou shows that this is as close as one can get to the critical level using iteration methods.\footnote{The discussion of the first iterate can be also found in the appendix of Klainerman and Selberg \cite{KS}, and it can be deduced from the estimates and counterexamples found within Foschi and Klainerman \cite{FK}.}  On the other hand, for dimensions $n \geq 3$ Klainerman and Machedon \cite{KM3} showed almost optimal local wellposedness in $H^s \times H^{s-1}$ for $s>\frac{n}{2}$.  Work of Klainerman and Machedon \cite{KlainermanMachedon95} and Klainerman and Selberg \cite{KS} gives as satisfying results for $Q_0$, and in all dimensions $n \geq 2$.\\
\indent
Now, one of the nonlinearities in the system (aME) is $Q_{\alpha\beta}$, so showing (aME) is locally wellposed for $s>\frac{5}{4}$ would be sharp by iteration methods.  This is what we do, and as a result we obtain local wellposedness of (ME) in the Coulomb gauge for $s>\frac{1}{4}$ (see Main Theorem below).  However, (aME) is not exactly (ME), so we hope to treat (ME) directly in the near future and improve the results.  What should be mentioned here is that we have considered other traditional gauges such as Lorentz and Temporal, but they have not been as nearly useful as the Coulomb gauge.  Perhaps other, less traditional gauges could be used.  Moreover, we note that even the estimates involving the temporal variable $A_0$ seem to require $s>\frac{1}{4}$.
\\   
\indent 
The main result of this paper is contained in the following theorem.
\begin{main}
Let $\frac{1}{4}<s<\frac{1}{2}$ and $r \in (0,2s]$ and consider the space-time monopole equation
\be\tag{ME}
 F_A=\ast D_A\phi,  
\ee
with initial data
\[
(A_1, A_2,\phi)|_{t=0}=(a_1,a_2,\phi_0),
 \]
then (ME) in the Coulomb gauge is locally wellposed for initial data sufficiently small in $H^s(\R^2)$ in the following sense:
\begin{itemize}
\item \textbf{{(Local Existence)}} 
For all $a_1,a_2,\phi_0 \in H^s(\R^2)$ sufficiently small there exist $T>0$ depending continuously on the norm of the initial data, and functions 
\begin{align*}
A_0 \in C_b([0,T],\dot H^{r}),\quad A_1,A_2,\phi \in C_b([0,T],H^s),
\end{align*}
which solve (ME) in the Coulomb gauge on $[0,T] \times \R^2$ in the sense of distributions and such that the initial conditions are satisfied.
\item \textbf{(Uniqueness)} If $T>0$ and $(A,\phi)$ and $(A',\phi')$ are two solutions of (ME) in the Coulomb gauge on $(0,T)\times \R^2$ belonging to 
\[
C_b([0,T],\dot H^{r})\times(H^{s,\theta}_T)^3
\] 
with the same initial data, then $(A,\phi)=(A',\phi')$ on $(0,T)\times \R^2$.
\item \textbf{(Continuous Dependence on the Initial Data)} For any $a_1,a_2,\phi_0 \in H^s(\R^2)$ there is a neighborhood $U$ of $a_1,a_2,\phi_0$ in $(H^s(\R^2))^3$ such that the solution map $(a,\phi_0) \rightarrow (A,\phi)$ is continuous from $U$ into $C_b([0,T],\dot H^{r})\times (C_b([0,T], H^s))^3$.  
\end{itemize}
\end{main}
\begin{remark}
Spaces $H^{s,\theta}_T$ are defined in Section \ref{spaces}.
\end{remark}
\begin{remark}
There are two reasons for the requirement of the small initial data.  First, the construction of the global Coulomb gauge requires an assumption on the size of the data (see Section \ref{gaugesection}).  The second obstacle comes from the elliptic equation for $A_0$ in (aME), and including $A_0$ in the Picard iteration.  See Remark \ref{comparison} for further discussion.
\end{remark}
\begin{remark}
We do not prescribe initial data for $A_0,$ because when $A$ is in the Coulomb gauge, $A_0(t)$ can be determined at any time by solving the elliptic equation.  See Section \ref{mesys} for more details.
\end{remark}
\begin{remark}
For technical reasons involving estimates for $A_0$ and the regularity of the gauge transformations, in this paper we assume $\frac{1}{4}<s<\frac{1}{2}$.  See \cite{Czubak} for all $s>\frac{1}{4}$.
\end{remark}
The outline of the paper is as follows.  Section \ref{prelim} sets notation, introduces spaces, and estimates used throughout the paper.  In Section \ref{closer} we take a closer look at the equations and discuss gauge transformations.  
In Section \ref{mesys} we rewrite (ME) as a system of wave equations coupled with an elliptic equation.  We also show local wellposedness of the new system implies local wellposedness of (ME) in the Coulomb gauge.  Section \ref{main} is devoted to the proof of the Main Theorem, which is reduced to establishing estimates \eqref{M1}-\eqref{ell}.  \begin{acknowledgment}
The author would like to express deep gratitude to her thesis advisor Karen Uhlenbeck for her time, many helpful discussions, and in particular for suggesting the problem with the reformulation \eqref{system1}-\eqref{system2}.
\end{acknowledgment}
\section{Preliminaries}\label{prelim}
First we establish notation, then we introduce function spaces as well as estimates used.  
\subsection{Notation}
$a \lesssim b$ means $a \leq Cb$ for some positive constant $C$.  A point in the $2+1$ dimensional Minkowski space is written as $(t,x)=(x^\alpha)_{0 \leq \alpha \leq 2}.$  Greek indices range from $0$ to $2$, and Roman indices range from $1$ to $2$.  We raise and lower indices with the Minkowski metric $\diag(-1,1,1)$.  We write $\partial_\alpha=\partial_{x^\alpha}$ and $\partial_t=\partial_0$, and we also use the Einstein notation.  Therefore, $\partial^i \partial_i =\triangle,$ and $\partial^\alpha \partial_\alpha=-\partial^2_t +\triangle=\square$.  When we refer to spatial and time derivatives of a function $f$, we write $\partial f$, and when we consider only spatial derivatives of $f$, we write $\nabla f$.  Finally, $d$ denotes the exterior differentiation operator and $d^\ast$ its dual given by $d^\ast=(-1)^k\ast\ast\ast d \ast$, where $\ast$ is the Hodge $\ast$ operator (see for example \cite{Roe}) and $k$ comes from $d^\ast$ acting on some given $k$-form.  It will be clear from the context, when $\ast$ and $d^\ast$ operators act with respect to the Minkowski metric and when with respect to the euclidean metric.  For the convenience of the reader we include the following: with respect to the euclidean metric on $\R^2$ we have 
\[
\ast dx=dy,\quad \ast dy=-dx,\quad \ast 1=dx\wedge dy,
\]
and with respect to the $\diag(-1,1,1)$ metric on $\R^{2+1}$
\[
\ast dt=dx\wedge dy,\quad \ast dx=dt\wedge dy,\quad \ast dy=-dt\wedge dx.
\]
\subsection{Function Spaces}\label{spaces}
We use Picard iteration.  Here we introduce the spaces, in which we are going to perform the iteration\footnote{We are also going to employ a combination of the standard $L^p_tW^{s,q}_x$ spaces for $A_0$.  See Section \ref{ellipticpiece1}.}.
First we define following Fourier multiplier operators
\begin{align}
\widehat{\Lambda^\alpha f}(\xi)&=(1 + |\xi|^2)^\frac{\alpha}{2}\hat f(\xi),\\
\widehat{\Lambda^\alpha_+ u}(\tau,\xi)&=(1 + \tau^2+|\xi|^2)^\frac{\alpha}{2}\hat u(\tau,\xi),\\
\widehat{\Lambda^\alpha_- u}(\tau,\xi)&=\left(1 +\frac{ (\tau^2-|\xi|^2)^2}{1+\tau^2+|\xi|^2}\right)^\frac{\alpha}{2}\hat u(\tau,\xi),
\end{align}
where the symbol of $\Lambda^\alpha_-$ is comparable to $(1 + \big||\tau|-|\xi|\big|)^\alpha$.  The corresponding homogeneous operators are denoted by $D^\alpha,D_+^\alpha,D_-^\alpha$ respectively.
\newline\indent Now, the spaces of interest are the Wave-Sobolev spaces, $H^{s,\theta}$ and $\H$, given by\footnote{These spaces, together with results in \cite{SelbergEstimates}, allowed Klainerman and Selberg to present a unified approach to local wellposedness for Wave Maps, Yang-Mills and Maxwell-Klein-Gordon types of equations in \cite{KS}, and are now the natural choice for low regularity subcritical local wellposedness for wave equations.  Also see \cite{taobook}.}
\begin{align}
\|u\|_{H^{s,\theta}}&=\| \Lambda^s\Lambda^\theta_-u \|_{L^2(\R^{2+1})},\\
\|u\|_\H&=\|u\|_{H^{s,\theta}}+\|\partial_tu\|_{H^{s-1,\theta}}.
\end{align}
An equivalent norm for $\H$ is $\|u\|_{\H}=\| \Lambda^{s-1}\Lambda_+\Lambda^\theta_-u \|_{L^2(\R^{2+1})}$.  By results in \cite{SelbergT} if $\theta > \frac{1}{2}$, we have
\begin{eqnarray}
H^{s,\theta} &\hookrightarrow & C_b(\R,H^s),\label{embed}\\
\H &\hookrightarrow & C_b(\R,H^s)\cap C_b^1(\R,H^{s-1}).
\end{eqnarray}
This is a crucial fact needed to localize our solutions in time.  We denote the restrictions to the time interval $[0,T]$ by
\[
 H^{s,\theta}_T\qand \mathcal H^{s,\theta}_T,
\]
respectively.
\subsection{{Estimates Used}}
Throughout the paper we use the following estimates:
\begin{align}
\|D^{-\sigma} (uv) \|_{L^p_tL^q_x} &\lesssim  \|u\|_{H^{s,\theta}} \|v\|_{H^{s,\theta}},\label{A}\\
\|u\|_{L^{p}_tL^2_x} &\lesssim  \|u\|_{H^{0,\theta}},\quad 2\leq p \leq \infty,\ \theta >\frac{1}{2},\label{C}\\
\|u\|_{L^{p}_tL^q_x} &\lesssim  \|u\|_{H^{1-\frac{2}{q}-\frac{1}{p},\theta}},\quad 2\leq p \leq \infty,\ 2\leq q < \infty,\ \frac{2}{p}\leq \frac{1}{2}-\frac{1}{q},\ \theta>\frac{1}{2},\label{D}\\
\|uv\|_{L^2_{t,x}}&\lesssim\|u\|_{H^{a,\alpha}}\|v\|_{H^{b,\beta}},\quad a,b,\alpha,\beta\geq0,\ a+b>1,\ \alpha+\beta>\frac{1}{2}.\label{E}
\end{align}
Estimate \eqref{A} is a theorem of Klainerman and Tataru established in \cite{KlainermanTataru} for the space-time operator $D_+$.  The proof for the spatial operator $D$ was shown by Selberg in \cite{SelbergT}.  There are several conditions $\sigma, p,q$ have to satisfy, and they are listed in Section \ref{KTT}, where we discuss the application of the estimate.  
Estimate \eqref{C} can be proved by interpolation between $H^{0,\theta} \hookrightarrow L^2_{t,x}$ and \eqref{embed} with $s=0$.  \eqref{D} is a two dimensional case of Theorem D in Klainerman and Selberg \cite {KS}.  Finally, \eqref{E} is a special case of the proposition in Appendix A.2 also in \cite{KS}.
\end{section} 
\begin{section}{A Closer Look at the Monopole Equations}\label{closer}
\begin{subsection}{Derivation and Background}
Electric charge is quantized, which means that it appears in integer multiples of an electron.  This is called the principle of quantization and has been observed in nature.  The only theoretical proof so far was presented by Paul Dirac in 1931 \cite{Dirac}.  In the proof Dirac introduced the concept of a magnetic monopole, of an isolated point-source of a magnetic charge.  Despite extensive research magnetic monopoles have not been (yet) found in nature. We refer to the magnetic monopoles as euclidean monopoles.  The euclidean monopole equation has exactly the same form as our space-time monopole equation (ME),
\[
F_A=\ast D_A\phi,
\]
with the exception that $\ast$ acts here with respect to the euclidean metric and the base manifold is $\R^3$ instead of $\R^{2+1}$.  The euclidean monopole equations are also referred to as Bogomolny equations.  For more on euclidean monopoles we refer an interested reader to books by Jaffe and Taubes \cite{JaffeTaubes} and Atiyah and Hitchin \cite{AtiyahHitchin}.  In this paper we study the space-time monopole equation, which was first introduced by Ward \cite{Ward89}.  Both the euclidean and the space-time monopole equations are examples of integrable systems and have an equivalent formulation as a Lax pair.  This and much more can be found in \cite{DCU}.\\
\indent
Given a space-time monopole equation
\be\tag{ME}
F_A=\ast D_A\phi,
\ee
the unknowns are a pair $(A,\phi)$. $A$ is a connection $1$-form given by 
\be
A=A_0dt+A_1dx+A_2dy,\quad\mbox{where}\quad A_\alpha: \R^{2+1} \rightarrow \mathfrak{g}.
\ee
$\mathfrak{g}$ is the Lie algebra of a Lie group $G$, which is typically taken to be a matrix group $SU(n)$ or $U(n)$.  In this paper we consider $G=SU(n),$ but everything we say here should generalize to any compact Lie group.\\  
\indent
To be more general we could say $A$ is a connection on a principal G-bundle.  Then observe that the G-bundle we deal here with is a trivial bundle $\R^{2+1} \times G$.  
\newline \indent
Next, $\phi$ is a section of a vector bundle associated to the G-bundle by a representation.  We use the adjoint representation.  Since we have a trivial bundle, we can just think of the Higgs field $\phi$ as a map from $\R^{2,1} \rightarrow \mathfrak{g}$.\\
\indent
$F_A$ is the curvature of $A$.  It is a Lie algebra valued $2$-form on $\R^{2+1}$
\be
F_A=\frac{1}{2}F_{\alpha\beta}dx^\alpha\wedge dx^\beta,\quad\mbox{where}\quad F_{\alpha\beta}=\partial_\alpha A_\beta -\partial_\beta A_\alpha +[A_\alpha,A_\beta]. \label{curv}
\ee
$[\cdot,\cdot]$ denotes the Lie bracket, which for matrices can be thought of simply as $[X,Y]=XY-YX$.  When we write $[\phi,B]$, where $B$ is a $1$-form, we mean
\be
[\phi,B]=[\phi,B_i]dx^i\qand [B,C]=\frac{1}{2}[B_i,C_j]dx^i\wedge dx^j,\mbox{ for two $1$-forms $B,C$.}
\ee
In the physics language, frequently adopted by the mathematicians, $A$ is called a gauge potential, $\phi$ a scalar field and $F_A$ is called an electromagnetic field.
\newline\indent Next, $D_A$ is the covariant exterior derivative associated to $A$, and $D_A\phi$ is given by
\be
D_A \phi=D_\alpha \phi dx^\alpha,\quad\mbox{where}\quad D_\alpha\phi=\partial_\alpha \phi + [A_\alpha,\phi].
\ee
\indent    
The space-time monopole equation (ME) is obtained by a dimensional reduction of the anti-self-dual Yang Mills equations on $\R^{2+2}$ given by
\be\tag{ASDYM}
F_A=-\ast F_A.
\ee
We now present the details of the derivation of (ME) from (ASDYM) outlined in \cite{DCU}.  Let 
\[
dx_1^2 + dx_2^2 - dx_3^2 - dx_4^2
\]
be a metric on $\R^{2+2}$, then in coordinates (ASDYM) is
\be
F_{12}=-F_{34},\quad F_{13}=-F_{24},\quad F_{23}=F_{14}.\label{asdymcoor}
\ee
Next step is the dimensional reduction, where we assume the connection $A$ is independent of $x_3$, and we let $A_3=\phi$.  Then \eqref{asdymcoor} becomes
\be
D_0 \phi=F_{12}, \quad D_1 \phi=F_{02}, \quad D_2 \phi=F_{10},
\ee
where we use index $0$ instead of $4$. This is exactly (ME) written out in components.
\begin{remark}
Equivalently we could write (ME) as
\be
F_{\alpha\beta}=-\epsilon_{\alpha\beta\gamma}D^\gamma \phi,
\ee
where $\epsilon_{\alpha\beta\gamma}$ is a completely antisymmetric tensor with $\epsilon_{012}=1$, and where we raise the index $\gamma$ using the Minkowski metric.  We choose to work with the Hodge operator $\ast$ as it simplifies our task in Section \ref{mesys}.
\end{remark}
\indent
There is another way to write (ME) \cite{DCU}, which is very useful for computations.  (ME) is an equation involving $2$-forms on both sides.  By taking the parts corresponding to $dt\wedge dx$ and $dt\wedge dy$, and the parts corresponding to $dx\wedge dy$ we can obtain the following two equations respectively
\begin{align}
\partial_t A + [A_0, A]-dA_0&=\ast d\phi + [\ast A, \phi],\label{m1a}\\
dA + [A,A]&=\ast(\partial_t \phi + [A_0, \phi]).\label{m2a}
\end{align}
Observe that now operators $d$ and $\ast$ act only with respect to the spatial variables.  Similarly, $A$ now denotes only the spatial part of the connection, i.e., $A=(A_1,A_2)$.  Moreover, \eqref{m1a} is an equation involving $1$-forms, and \eqref{m2a} involves $2$-forms.\\
\end{subsection}
\subsection{Gauge Transformations}\label{gaugesection}
(ME) is invariant under gauge transformations.  Indeed, if we have a smooth map $g$, with compact support such that $g: \R^{2+1} \rightarrow G$, and
\begin{align}
A \rightarrow A_g&=gAg^{-1} + gdg^{-1},\\
\phi \rightarrow \phi_g&=g\phi g^{-1},
\end{align}
then a computation shows $F_A \rightarrow g F_A g^{-1}$ and $D_A\phi \rightarrow gD_A\phi g^{-1}$.  Therefore if a pair $(A,\phi)$ solves (ME), so does $(A_g,\phi_g)$.\\
\indent 
We would like to discuss the regularity of the gauge transformations.  If $A\in X, \phi \in Y$ where $X, Y$ are some Banach spaces, the smoothness and compact support assumption on $g$ can be lowered just enough so the gauge transformation defined above is a continuous map from $X$ back into $X$, and from $Y$ back into $Y$.  First note that since we are mapping into a compact Lie group, we can assume $g\in L^\infty_{t,x}$ and $\|g\|_{L^\infty_{t,x}}=\|g^{-1}\|_{L^\infty_{t,x}}.$  Next, note that the Main Theorem produces a solution so that $\phi$ and the spatial parts of the connection $A_1,A_2 \in C_b(I,H^s), \frac{1}{4}<s<\frac{1}{2}$, and $A_0 \in C_b(I,\dot H^r), r \in (0, 2s]$.  We have the following 
\lm\label{gaugeaction1}
Let $0<\alpha<1$, and $Y=C_b(I,\dot H^{1}\cap\dot H^{\alpha+1})\cap L^\infty,$ then the gauge action is a continuous map from
\be
\bs
C_b(I, H^\alpha) &\times Y  \rightarrow C_b(I, H^\alpha)\\
&(h,g) \mapsto ghg^{-1} + gdg^{-1},
\end{split}
\ee
and the following estimate holds:
\be
\|h_g\|_{C_b(I, H^\alpha)} \lesssim (\|h\|_{C_b(I, H^\alpha)}+1)\|g\|^2_Y.\label{hg}
\ee
\ml
\begin{proof}
The continuity of the map easily follows from the inequalities we obtain below.  Next, for fixed $t$ we have
\[
\|g(t)h(t)g^{-1}(t)+g(t)dg^{-1}(t)\|_{H^\alpha}\lesssim \|ghg^{-1}\|_{L^2}+\|D^\alpha(ghg^{-1})\|_{L^2} + \|gdg^{-1}\|_{H^\alpha},
\]
where for the ease of notation we eliminated writing of the variable $t$ on the right hand side of the inequality.  The first term is bounded by $\|h(t)\|_{H^\alpha}\|g\|_{L^\infty}^2$.  For the second one we have
\be\nonumber
\|D^\alpha(ghg^{-1})\|_{L^2} \lesssim \|D^\alpha gh\|_{L^2}\|g\|_{L^\infty}+\|hD^\alpha g^{-1}\|_{L^2}\|g\|_{L^\infty} +\|h\|_{\dot H^\alpha}\|g\|^2_{L^\infty}.
\ee
It is enough to only look at the first term since $g$ and $g^{-1}$ have the same regularity.  By H\"older's inequality and Sobolev embedding 
\be
\|D^\alpha gh\|_{L^2}\leq \|D^\alpha g\|_{L^{2/\alpha}}\|h\|_{L^{(1/2-\alpha/2)^{-1}}}\lesssim \|g\|_{\dot H^1}\|h\|_{\dot H^\alpha},\label{gA}
\ee
where we use that $\frac{\alpha}{2}=\frac{1}{2}-\frac{1-\alpha}{2}.$
Finally for the last term we have
\be
\|gdg^{-1}\|_{H^\alpha}\lesssim \|g\|_{\dot H^1}\|g\|_{L^\infty}+\|D^\alpha gdg^{-1}\|_{L^2}+\|g\|_{\dot H^{\alpha+1}}\|g\|_{L^\infty},
\ee
and we are done if we observe that the second term can be handled exactly as in \eqref{gA}.
\end{proof}
\begin{remark}
We assume $0<\alpha<1$ since this is the case we need.  However it is not difficult to see the lemma still holds with $\alpha=0$ or $\alpha \geq 1$;  see \cite{Czubak}.
\end{remark}
From the lemma, we trivially obtain the following corollary.
\begin{cor}\label{welld} 
Let $0<r,s<1$, $X=C_b(I,\dot H^r)\times C_b(I, H^s)\times C_b(I,H^s)$ and
$Y=C_b(I,\dot H^{1}\cap \dot H^{s+1}\cap \dot H^{r+1})\cap L^\infty.$  Then the gauge action is a continuous map from
\be
\bs
X \times Y &\rightarrow X\\
(A_0,A_1,A_2) 							& \mapsto A_g,
\end{split}
\ee
as well as from
\be
\bs
C_b(I,H^s) \times Y &\rightarrow C_b(I,H^s)\\
\phi 							& \mapsto \phi_g=g\phi g^{-1}, 
\end{split}
\ee
and the following estimates hold
\begin{align}
\|A_g\|_X \lesssim (1+\|A\|_X)\|g\|^2_Y,\\
\|\phi_g\|_{C_b(I,H^s)} \lesssim \|\phi\|_{C_b(I,H^s)}\|g\|^2_Y.
\end{align}
\end{cor}
Since in this paper we work in the Coulomb gauge, we ask: given any initial data $a_1, a_2, \phi_0 \in H^s(\R^2)$, can we find a gauge transformation so that the initial data is placed in the Coulomb gauge?  Dell'Antonio and Zwanziger produce a global $\dot H^1$ Coulomb gauge using variational methods \cite{DZ}.  Here, we also require $g \in \dot H^{s+1}$, and two dimensions are tricky.  Fortunately, if the initial data is small, we can obtain a global gauge with the additional regularity as needed.  This is considered by the author and Uhlenbeck for two dimensions and higher in \cite{CzubakUhlenbeck}.  The result in two dimensions is the following
\begin{thm}(\cite{CzubakUhlenbeck})\label{coulombg}
Given $A(0)=a$ sufficiently small in $H^s(\R^2)\times H^s(\R^2)$, there exists a gauge transformation $g \in \dot H^{s+1}(\R^2) \cap \dot H^{1}(\R^2)\cap L^\infty$ so that $\partial^i(ga_ig^{-1}+g\partial_ig^{-1})=0$.
\end{thm}
\end{section}



\section{The Monopole Equation in the Coulomb Gauge as a system of Wave \& Elliptic Equations}\label{mesys}
We begin with a proposition, where we show how we can rewrite the monopole equation in the Coulomb gauge as a system of wave equations coupled with an elliptic equation, to which we refer as the auxiliary monopole equations (aME).  Then we have an important result that states that local wellposedness (LWP) for (ME) in the Coulomb gauge can be obtained from LWP of (aME). 
\begin{prop}
The monopole equation, $F_A=\ast D_A \phi$ on $\R^{2+1}$ in the Coulomb gauge with initial data  
\be\label{id0}
A_i|_{t=0}=a_i,\quad i=1,2 \quad\mbox and \mbox\quad \phi|_{t=0}=\phi_0  
\ee
with $\partial^i a_i=0$ can be rewritten as the following system
\[ \mbox{(aME)}\quad
   \left\{ \begin{array}{l} 
\begin{split}
\square u     &= \mathcal B_+(\phi,\nabla f,A_0),\nonumber\\
\square v     &= \mathcal B_-(\phi,\nabla f,A_0),\nonumber\\
\triangle A_0 &=\mathcal C(\phi,\nabla f,A_0),\nonumber
\end{split}
\end{array} \right. \]
where
\begin{align}
&\mathcal C=-\partial_1[A_0,\partial_2 f]+\partial_2[A_0,\partial_1 f]+\partial_i[\partial_if,\phi],\\
&\mathcal B_\pm=-\mathcal B_1 \pm\mathcal B_2+\mathcal B_3\pm\mathcal B_4,\label{Bpm}
\end{align}
and
\begin{equation}\label{bis}
\begin{split}
\mathcal B_1&=[\partial_1f,\partial_2f],\\
\mathcal B_2&=R_1[\partial_2f,\phi]-R_2[\partial_1f,\phi],\\
\mathcal B_3&=[A_0,\phi],\\
\mathcal B_4&=R_j[A_0,\partial_j f],
\end{split}
\end{equation}
with $R_j$ denoting Riesz transform, $(-\triangle)^{-\frac{1}{2}}\partial_j.$
The initial data for (aME) is given by
\begin{equation}\label{id}
\begin{split}
u(0)&=v(0)=0,\\
\partial_t u(0)&=\phi_0 +  h,\\ 
\partial_t v(0)&=\phi_0 -  h, 
\end{split}
\end{equation}
where $h=R_1 a_2-R_2 a_1$.
\end{prop}
\begin{remark}\label{comparison}
(aME) has some resemblance to a system considered by Selberg in \cite{Selberg} for the Maxwell-Klein-Gordon (MKG) equations, where he successfully obtains almost optimal local wellposedness in dimensions $1+4$.  Besides the dimension considered, there are two fundamental technical differences applicable to our problem. First comes from the fact that the monopole equation we consider here is an example of a system in the non-abelian gauge theory whereas (MKG) is an example of a system in the abelian gauge theory.  The existence of a global Coulomb gauge requires smallness of initial data in the non-abelian gauge theories, but is not needed in the abelian theories.  Another technical difference arises from Selberg being able to solve the elliptic equation for his temporal variable $A_0$ using Riesz representation theorem, where he does not require smallness of the initial data.  The elliptic equation in (aME) is more difficult, so we include $A_0$ in the Picard iteration.  As a result we are not able to allow large data by taking a small time interval, which we could do if we only had the two wave equations.  Finally, we point out that the proof of our estimates involving $A_0$ is modeled after Selberg's proof in \cite{Selberg} (see Remark \ref{estinMKG} and Section \ref{ellipticpiece1}).
\end{remark}
\begin{proof}
Recall equations \eqref{m1a} and \eqref{m2a}
\begin{eqnarray}
\partial_t A + [A_0, A]-dA_0=\ast d\phi + [\ast A, \phi]\label{m1},\\
dA + [A,A]=\ast(\partial_t \phi + [A_0, \phi]),\label{m2}
\end{eqnarray}
where $d$ and $\ast$ act only with respect to the spatial variables, and $A$ denotes only the spatial part of the connection.
If we impose the Coulomb gauge condition, then
\be
d^\ast A=0.
\ee
By equivalence of closed and exact forms on $\R^n$, we can further suppose that
\be\label{aisdf}
A=\ast df,
\ee  
for some $f:\R^{2+1} \rightarrow \mathfrak g$.  Observe 
\be d \ast d f=\triangle f dx\wedge dy, \quad [\ast df, \ast df]=[df,df]=\frac{1}{2}[\partial_if,\partial_jf]dx^i\wedge dx^j,
\ee
and $\ast\ast \omega=-\omega$ for a one-form on $\R^2$.  It follows \eqref{m1} and \eqref{m2} become
\begin{eqnarray}
\partial_t \ast df + [A_0, \ast df]-dA_0=\ast d\phi - [df, \phi],\label{e1'}\\
\triangle f + [\partial_1f, \partial_2f]=\partial_t \phi + [A_0, \phi].\label{e2'}
\end{eqnarray}
Take $d^\ast$ of \eqref{e1'} to obtain
\[
\triangle A_0= d^\ast [A_0, \ast df]+d^\ast[df, \phi].
\]
This is the elliptic equation in (aME).  Now take $d$ of \eqref{e1'}
\be
\partial_t \triangle f + \partial^j[A_0, \partial_jf]=\triangle \phi + \partial_2[\partial_1f,\phi]- \partial_1[\partial_2f,\phi].\label{e1'''}
\ee
Consider \eqref{e1'''} and \eqref{e2'} on the spatial Fourier transform side
\begin{align}
-\partial_t |\xi|^2 \hat f +|\xi|^2 \hat \phi & = i (\xi_2\widehat{[\partial_1f,\phi]}- \xi_1\widehat{[\partial_2f,\phi]} -\xi_j\widehat{[A_0, \partial_jf]}) \label{eI}\\
-|\xi|^2 \hat f -\partial_t \hat \phi &=- \widehat{[\partial_1f, \partial_2f]} + \widehat{[A_0, \phi]}.\label{eII}
\end{align}
This allows us to write \eqref{eI} and \eqref{eII} as a system for $\phi$ and $df$
\begin{eqnarray}
(\partial_t - i |\xi|)(\hat \phi + i |\xi|\hat f)= -\hat \B_{+}(\phi,df,A_0),\label{system1}\\
(\partial_t + i |\xi|)(\hat \phi - i |\xi|\hat f)= -\hat \B_{-}(\phi,df,A_0),\label{system2}
\end{eqnarray}
where
\be
\hat \B_{\pm}=-\widehat{[\partial_1f, \partial_2f]} + \widehat{[A_0,\phi]} \pm \left(\frac{\xi_1}{|\xi|}\widehat{[\partial_2f,\phi]}-\frac{\xi_2}{|\xi|}\widehat{[\partial_1f,\phi]}+\frac{\xi_j}{|\xi|}\widehat{[A_0, \partial_jf]}\right).
\ee
Indeed, multiply \eqref{eI} by $\frac{i}{|\xi|}$, and first add the resulting equation to \eqref{eII} to obtain \eqref{system1}, and then subtract it from \eqref{eII} to obtain \eqref{system2}.  To uncover the wave equation, we let 
\be
\hat \phi + i |\xi|\hat f=(\partial_t + i|\xi|)\hat u \qand  \hat \phi - i |\xi|\hat f=(\partial_t - i|\xi|)\hat v\label{two},
\ee
where $u,v: \R^{2+1} \rightarrow \mathfrak g$.  See remark \ref{iremark} below.\\    
\indent Now we discuss initial data.  
From (\ref{two}) 
\be
\partial_t \widehat {u(0)}=\hat \phi_0 + i|\xi|\widehat {f(0)} - i|\xi|\widehat {u(0)},
\ee
and
\be
\partial_t \widehat {v(0)}=\hat \phi_0 - i|\xi|\widehat {f(0)} + i|\xi|\widehat {v(0)}.
\ee
Note, we are free to choose any data for $u$ and $v$ as long as in the end we can recover the original data for $\phi$ and $A$.  Hence we just let $u(0)=v(0)=0$.  We still need to say what $|\xi|\widehat{f(0)}$ is.  Let $\hat h=i|\xi|\widehat {f(0)}$.  By \eqref{id0} and \eqref{aisdf} 
\[
a_1=A_1(0)=-\partial_2 f(0),\quad a_2=A_2(0)=\partial_1 f(0),
\]
so we need
\[
R_2h=-a_1,\quad R_1h=a_2.
\]  
Differentiate the first equation with respect to $x$, the second with respect to $y$, and add them together to obtain
\be
\triangle D^{-1} h=\partial_1 a_2 - \partial_2 a_1,
\ee
as needed.
\end{proof}
\begin{remark}\label{iremark}
$u$ and $v$ are our new unknowns,  but we are really interested in $\phi$ and $df$.  Therefore, we observe that once we know what $u$ and $v$ are, we can determine $\phi$ and $df$ by using
\be\label{phidf}
\begin{split}
\hat \phi=\frac{(\partial_t + i|\xi|)\hat u + (\partial_t - i|\xi|)\hat v}{2},\\
i|\xi| \hat f=\frac{(\partial_t + i|\xi|)\hat u - (\partial_t - i|\xi|)\hat v}{2},
\end{split}   
\ee
or equivalently
\be\label{newphidf}
\bs
\phi&=\frac{(\partial_t + iD)u+(\partial_t - iD)v}{2},\\
\partial_j f&=R_j\left(\frac{(\partial_t + iD)u-(\partial_t - iD)v}{2}\right).
\end{split}
\ee
From $df$ we get $A$ by letting $A=\ast df$.  Finally, with the exception of the nonlinearity $\mathcal B_2$ when we discuss our estimates in Section \ref{main}, for simplicity we keep the nonlinearities in terms of $\phi$ and $df$.
However, since $\phi$ and $df$ can be written in terms of derivatives of $u$ and $v$ we sometimes write $\B_\pm(\phi,df,A_0)$ as $\B_\pm(\partial u,\partial v,A_0)$. 
\end{remark}
Next we have a theorem, where we show how LWP for (aME) implies LWP for (ME) in the Coulomb gauge.  For completeness, we first state exactly what we mean by LWP of (aME).\\
\indent Let $r\in (0,\min(2s,1+s)], s>0$.  Consider the system (aME) with initial data
\[
(u,u_t)|_{t=0}=(u_0,u_1)\qand(v,v_t)|_{t=0}=(v_0,v_1)
\]
in $H^{s+1}\times H^s$, then (aME) is LWP if:
\newline\noindent
\textit{\textbf{{(Local Existence)}}} 
There exist $T>0$ depending continuously on the norm of the initial data, and functions 
\begin{align*}
A_0 \in C_b([0,T],\dot H^{r}),\quad
u,v \in \mathcal H_T^{s+1,\theta}\hookrightarrow C_b([0,T],H^{s+1})\cap C^1_b([0,T],H^s),
\end{align*}
which solve (aME) on $[0,T] \times \R^2$ in the sense of distributions and such that the initial conditions are satisfied.
\newline\noindent
\textit{\textbf{(Uniqueness)}} If $T>0$ and $(A_0,u,v)$ and $(A_0',u',v')$ are two solutions of (aME) on $(0,T)\times \R^2$ belonging to 
\[
C_b([0,T],\dot H^{r})\times \mathcal H_T^{s+1,\theta}\times \mathcal H_T^{s+1,\theta},
\] 
with the same initial data, then $(A_0,u,v)=(A_0',u',v')$ on $(0,T)\times \R^2$.
\newline\noindent
\textit{\textbf{(Continuous Dependence on Initial Data)}} For any $(u_0,u_1), (v_0,v_1) \in H^{s+1}\times H^s$ there is a neighborhood $U$ of the initial data such that the solution map $(u_0,u_1),(v_0,v_1) \rightarrow (A_0,u,v)$ is continuous from $U$ into $C_b([0,T],\dot H^{r})\times \big(C_b([0,T],H^{s+1})\cap C^1_b([0,T],H^s)\big)^2$.\newline\indent  
In fact, by the results in \cite{SelbergEstimates} combined with estimates for the elliptic equation, we can show these stronger estimates
\be\label{cdimp}
\bs
\|u-&u'\|_{\mathcal H_T^{s+1,\theta}}+\|v-v'\|_{\mathcal H_T^{s+1,\theta}}+\|A_0-A_0'\|_{C_b([0,T],\dot H^{r})}\\
			         &\lesssim \|u_0-u_0'\|_{H^{s+1}}+\|u_1-u_1'\|_{H^{s}} +\|v_0-v_0'\|_{H^{s+1}}+\|v_1-v_1'\|_{H^{s}},
\end{split}
\ee
where $(u_0',u_1'),(v_0',v_1')$ are sufficiently close to $(u_0,u_1),(v_0,v_1)$.
\begin{remark}
Note that below we have no restriction on $s$, i.e., if we \emph{could} show (aME) is LWP in $H^{s+1}\times H^s,$ $s>0$, we would get LWP of (ME) in the Coulomb gauge in $H^s$ for $s>0$ as well. 
\end{remark}
\begin{thm}(\textbf{Return to the Monopole Equation})\label{returnME}
Consider (ME) in the Coulomb gauge with the following initial data in $H^s$ for $s>0$ 
\be\label{id1}
A_i|_{t=0}=a_i,\quad i=1,2 \quad\mbox {and} \mbox\quad \phi|_{t=0}=\phi_0  
\ee
with $\partial^i a_i=0$.
Then local wellposedness of (aME) with initial data as in \eqref{id} implies local wellposedness of (ME) in the Coulomb gauge with initial data given by \eqref{id1}.
\end{thm}
\begin{proof}
Begin by observing that given initial data in the Coulomb gauge, the solutions of (aME) imply $A$ remains in the Coulomb gauge.  Indeed, solutions of (aME) produce $df$ via \eqref{newphidf}, so we get $A=\ast df$, and $d^\ast A=\ast d \ast(\ast df)=0$ as claimed.
\newline\noindent
\textit{\textbf{(Local Existence)}}
From \eqref{newphidf}, if
\[
 u,v \in \mathcal H_T^{s+1,\theta},\quad\mbox{then}\quad \phi, A=\ast df \in H_T^{s,\theta},
\]
as needed.  Next we verify that if $(u,v,A_0)$ solve (aME), then $(\phi,df,A_0)$ solve (ME) in the Coulomb gauge.  Note, the monopole equation in the Coulomb gauge is equivalent to \eqref{e1'} and \eqref{e2'}.   Suppose $u,v,A_0$ solve (aME).  It follows $(df,\phi)$ solve
(\ref{system1}) and (\ref{system2}).  Add (\ref{system1}) to (\ref{system2}) to recover (\ref{eII}), which is equivalent to \eqref{e2'}.\\
\indent
Next given (aME) we need to show \eqref{e1'} holds.  Write \eqref{e1'} in coordinates,
\begin{eqnarray}
\partial_1 A_0 -\partial_2\phi +\partial_t\partial_2f=[\partial_1f,\phi]-[A_0,\partial_2f],\label{something}\\
\partial_2 A_0+\partial_1\phi -\partial_t\partial_2f=[\partial_2f,\phi]+[A_0,\partial_1f].\label{somethingelse}
\end{eqnarray}
From the elliptic equation in (aME) we have 
\be
\begin{split}\label{anot}
A_0=\triangle^{-1}(-\partial_1[A_0,\partial_2f]+\partial_2[A_0,\partial_1f]
									 +\partial_1[\partial_1f,\phi]+\partial_2[\partial_2f,\phi]).
\end{split}
\ee
Also subtract (\ref{system1}) from (\ref{system2}) and multiply by $|\xi|$ on both sides to obtain \eqref{e1'''}, which implies
\be\label{difference}
\begin{split}
\phi-\partial_tf=\triangle^{-1}(\partial_i[A_0,\partial_if]-\partial_2[\partial_1f,\phi]+\partial_1[\partial_2f,\phi]).
\end{split}
\ee
In order to recover \eqref{something}, first use \eqref{anot} to get $\partial_1 A_0$
\be\label{something2}
\partial_1 A_0=\triangle^{-1}(-\partial^2_1[A_0,\partial_2f]+\partial_1\partial_2[A_0,\partial_1f]
							 +\partial^2_1[\partial_1f,\phi]+\partial_1\partial_2[\partial_2f,\phi]).
\ee
Next use \eqref{difference} to get $\partial_2(\phi-\partial_tf)$
\be\label{something3}
\partial_2(\phi-\partial_tf)=\triangle^{-1}(\partial_2\partial_1[A_0,\partial_1f]+\partial_2^2[A_0,\partial_2f]
-\partial_2^2[\partial_1f,\phi]+\partial_2\partial_1[\partial_2f,\phi]),
\ee
and subtract it from \eqref{something2} to get \eqref{something} as needed. 
We recover \eqref{somethingelse} in the exactly same way.
\newline\noindent
\textit{\textbf{(Continuous Dependence on Initial Data)}}
We would like to show
\be\label{cdoid}
\bs
\|A_0-A_0'\|_{C_b([0,T],\dot H^{r})}+&\|A_1-A_1'\|_{H_T^{s,\theta}}+\|A_2-A_2'\|_{H_T^{s,\theta}}+\|\phi-\phi'\|_{H_T^{s,\theta}}\\
															  &\lesssim \|a_1-a_1'\|_{H^s}+\|a_2-a_2'\|_{H^s}+\|\phi_0-\phi'_0\|_{H^s}
\end{split}
\ee
for any $a_1',a_2',\phi_0'$ sufficiently close to $a_1,a_2,\phi_0$.
In view of LWP for (aME) with data given by
\[
u(0)=v(0)=0,\quad \partial_t u(0)=\phi_0 +  h,\qand \partial_t v(0)=\phi_0 -  h,\quad h=R_1 a_2-R_2 a_1,
\]
and by \eqref{cdimp} we have
\be\label{cdimp1}
\bs
\|u-&u'\|_{\mathcal H_T^{s+1,\theta}}+\|v-v'\|_{\mathcal H_T^{s+1,\theta}}+\|A_0-A_0'\|_{C_b([0,T],\dot H^{r})}\\
			         &\lesssim \|u_0'\|_{H^{s+1}}+\|\phi_0 +  h-u_1'\|_{H^{s}} +\|v_0'\|_{H^{s+1}}+\|\phi_0 -  h-v_1'\|_{H^{s}},
\end{split}
\ee
for all $u_0',v_0',u_1',v_1'$ satisfying
\be\label{cdimp2}
 \|u_0'\|_{H^{s+1}}+\|\phi_0 +  h-u_1'\|_{H^{s}} +\|v_0'\|_{H^{s+1}}+\|\phi_0 -  h-v_1'\|_{H^{s}}\leq\delta,
\ee
for some $\delta >0$.  In particular choose 
\be\label{cdimp3}
u_0'=v_0'=0,\quad u_1'=\phi_0'+h' \qand v_1'=\phi_0'-h',\quad h'=R_1 a_2'-R_2 a_1', 
\ee
such that
\be\label{cdimp4}
\bs
 \|\phi_0+h-\phi_0'-h'\|_{H^s}& + \|\phi_0-h-\phi_0'+h'\|_{H^s}\\
                &\quad \lesssim \|\phi_0-\phi_0'\|_{H^s} + \|R_1(a_2-a_2')\|_{H^s}+\|R_2(a_1-a_1')\|_{H^s}\\
		&\quad\leq \|\phi_0-\phi'_0\|_{H^s}+\|a_1-a_1'\|_{H^s}+\|a_2-a_2'\|_{H^s}\\
		&\quad \leq \delta.
\end{split}
\ee
Then by \eqref{cdimp1}-\eqref{cdimp4}, $\|A_0-A_0'\|_{C_b([0,T],\dot H^{r})}$ is bounded by the right hand side of \eqref{cdoid}.  Next observe
\begin{align*}
\|A_1-A_1'\|_{H_T^{s,\theta}} &\lesssim\|R_2(\partial_t+iD)(u-u')\|_{H_T^{s,\theta}}+\|R_2(\partial_t-iD)(v-v')\|_{H_T^{s,\theta}}\\
                & \leq \|u-u'\|_{\mathcal H_T^{s+1,\theta}} +  \|v-v'\|_{\mathcal H_T^{s+1,\theta}}.
\end{align*}
So again by \eqref{cdimp1}-\eqref{cdimp4} $\|A_1-A_1'\|_{H_T^{s,\theta}}$ is bounded by the right hand side of \eqref{cdoid}.  We bound the difference for $A_2$ and $\phi$ in a similar fashion.
\newline\noindent
\textit{\textbf{(Uniqueness)}}
By LWP of (aME), $A_0$ is unique in the required class.  We need to show $A$ and $\phi$ are unique in $H^{s,\theta}_T$.  However, by \eqref{cdoid} this is obvious.
\end{proof}

\begin{section}{Proof of the Main Theorem}\label{main}
By Theorem \ref{returnME} it is enough to show LWP for (aME).  We start by explaining how we are going to perform our iteration.
\begin{subsection}{{Set up of the Iteration}}
Equations (aME) are written for functions $u$ and $v$.  Nevertheless, functions $u$ and $v$ are only our auxiliary functions, and we are really interested in solving for $df$ and $\phi$. In addition, the nonlinearities $\mathcal B_\pm$ are a linear combination of $\mathcal B_i$'s, $i=1,2,3,4$ given by \eqref{bis}, and $\mathcal B_i$'s are written in terms of $\phi, df$ and $A_0$.  Also, when we do our estimates, it is easier to keep the $\B_i$'s in terms of $\phi$ and $df$ with the exception of $\B_2$, which we rewrite in terms of $\partial u$ and $\partial v$\footnote{See Section \ref{nullformq} for the details.}.  These comments motivate the following procedure for our iteration.  Start with $\phi_{-1}=df_{-1}=0$.  Then $\B_\pm\equiv 0$.  Solve the homogeneous wave equations for $u_0,v_0$ with the initial data given by \eqref{id}.  Then to solve for $df_0, \phi_0$, use \eqref{newphidf}.
Then feed $\phi_0$ and $df_0$ into the elliptic equation, 
\be
\triangle A_{0,0}=d^\ast([A_{0,0}, \ast df_0]+ [df_0, \phi_0]),
\ee
and solve for $A_{0,0}$.  Next we take $df_0, \phi_0$ and $A_{0,0}$ plug them into $\mathcal B_1, \mathcal B_3, \mathcal B_4$,  but rewrite $\mathcal B_2$ in terms of $\partial u_0, \partial v_0$.  We continue in this manner, so at the j'th step of the iteration, $j \geq 1$, we solve
\be\nonumber
\begin{split}
\square u_j&= -\B_1(\nabla f_{j-1})+ \B_2(\partial u_{j-1}, \partial v_{j-1}) +\B_3(A_{0,j-1},\phi_{j-1})+\B_4(A_{0,j-1},\nabla f_{j-1}),\\
\square v_j&= -\B_1(\nabla f_{j-1})-\B_2(\partial u_{j-1}, \partial v_{j-1})+\B_3(A_{0,j-1},\phi_{j-1})-\B_4(A_{0,j-1},\nabla f_{j-1}),\\
\triangle A_{0,j}&=d^\ast([A_{0,j}, \ast df_j]+ [df_j, \phi_j]).
\end{split}
\ee
\end{subsection}
\begin{subsection}{Estimates Needed}
The elliptic equation is discussed in Section \ref{ellipticpiece1}.  Therefore we begin by discussing the inversion of the wave operator in $\Hs$ spaces.  The main idea is that for the purposes of local in time estimates $\square^{-1}$ can be replaced with $\Lambda^{-1}_+\Lambda^{-1}_-$.  The first estimates, leading to wellposedness for small initial data, were proved by Klainerman and Machedon in \cite{KlainermanMachedon95}.  The small data assumption was removed by Selberg in \cite{SelbergEstimates}, where he showed that by introducing $\epsilon$ small enough in the invertible version of the wave operator, i.e., $\Lambda^{-1}_+\Lambda^{-1+\epsilon}_-$, we can use initial data as large as we wish\footnote{See also \cite{KS} Section 5 for an excellent discussion and motivation of the issues involved in the Picard iteration.}.  In \cite{SelbergEstimates} Selberg also gave a very useful, general framework for local wellposedness of wave equations, which reduces the proof of the Main Theorem to establishing the estimates below, for the nonlinearities $\mathcal B_\pm$, and to combining them with appropriate elliptic estimates from Section \ref{ellipticpiece1}.  The needed estimates for $\mathcal B_\pm$ are
\be
\|\Lambda^{-1}_+\Lambda^{-1+\epsilon}_-\mathcal B_\pm(\partial u, \partial v,A_{0})\|_\Hs \lesssim
										 																														\|u\|_\Hs + \|v\|_\Hs,\label{bound}
\ee										 																													
\be\label{lip}
\begin{split}
\|\Lambda^{-1}_+\Lambda^{-1+\epsilon}_- \bigl(\mathcal B_\pm(\partial u,\partial v,A_0)-\mathcal B_\pm(\partial u',&\partial v',A_0')\bigr)\|_\Hs \\
 																																																				&\lesssim \|u-u'\|_\Hs+\|v-v'\|_\Hs,
\end{split}
\ee
where the suppressed constants depend continuously on the $\Hs$ norms of $u,u',v,v'$.  
Since $\mathcal B_\pm$ are bilinear, (\ref{lip}) can follow from (\ref{bound}).  
In this paper small initial data is necessary\footnote{See Theorem \ref{coulombg} and Section \ref{ellipticpiece1}.}, so we do not need $\epsilon$, but we keep it to make the estimates general.  Let $\frac{1}{4}<s<\frac{1}{2}$ and set $\theta,\epsilon$ as follows
\begin{align*}
\frac{3}{4}-\frac{\epsilon}{2}<\theta\leq s+\frac{1}{2}-\epsilon,\qand \theta<1-\epsilon,\quad
0\leq \epsilon<\min\left( 2s-\frac{1}{2},\frac{1}{2}\right).
\end{align*}
Next
observe $\Lambda_+\Lambda_-^{1-\epsilon} \Hs=H^{s,\theta-1+\epsilon}$, as well as that
\[
\|\nabla f\|_\h, \|\phi\|_\h \lesssim \|u\|_\Hs + \|v\|_\Hs.
\]
Therefore, using \eqref{Bpm} and \eqref{bis}, it is enough to prove the following
\begin{align}
\|\mathcal B_1\|_{H^{s,\theta-1+\epsilon}}&=\|[\partial_1f,\partial_2f]\|_{H^{s,\theta-1+\epsilon}} \lesssim \|\nabla f\|_\h^2\label{M1}, \\
\|\mathcal B_2\|_{H^{s,\theta-1+\epsilon}}&\lesssim\|[\partial_jf,\phi]\|_{H^{s,\theta-1+\epsilon}} \lesssim \|\partial_j f\|_\h\|\phi\|_\h,\quad j=1,2, \label{M3}  \\
\|\mathcal B_3\|_{H^{s,\theta-1+\epsilon}}&\lesssim\|A_0\phi\|_{H^{s,\theta-1+\epsilon}} \lesssim \|A_0\| \|\phi\|_\h \label{M2},  \\
\|\mathcal B_4\|_{H^{s,\theta-1+\epsilon}}&\lesssim\|A_0 \partial_jf\|_{H^{s,\theta-1+\epsilon}} \lesssim \|A_{0}\| \|\partial_j f\|_\h, \quad j=1,2,\label{M4} 
\end{align}
where the norm we are using for $A_0$ is immaterial, mainly because we show in Section \ref{ellipticpiece1}, 
\be
\|A_0\|\lesssim \|\nabla f\|_\h\|\phi\|_\h\label{ell}.
\ee
A few remarks are in order.  Estimate \eqref{M1} corresponds to estimates for the null form $Q_{ij}$, and estimate \eqref{M3} gives rise to a new null form $Q$ (this is discussed in the next two sections).  $A_0$ in estimates \eqref{M2} and \eqref{M4} solves the elliptic equation in (aME), which results in a quite good regularity for $A_0$.  As a result, we do not have to look for any special structures to get \eqref{M2} and \eqref{M4} to hold, so we can drop the brackets, and also treat these estimates as equivalent since $\phi$ and $df$ exhibit the same regularity.  Finally, since Riesz transforms are clearly bounded on $L^2$, we ignore them in the estimates needed in \eqref{M3} and \eqref{M4}. The estimates \eqref{M1} and \eqref{M3} for the null forms are the most interesting.  Hence we discuss them first, and then we consider the elliptic terms.
 
\begin{subsubsection}{Null Forms--Proof of Estimate (\ref{M1})}\label{nullformqij}
$[\partial_1f,\partial_2f]$ has a structure of a null form $Q_{ij}:$
\[ 
[\partial_1f, \partial_2f]=\partial_1f\partial_2f-\partial_2f\partial_1f=Q_{12}(f,f).
\]
It follows \eqref{M1} is equivalent to
\[
\|Q_{12}(f,f)\|_{H^{s,\theta-1+\epsilon}} \lesssim \|\nabla f\|_\h^2.
\]
Fortunately the hard work for null forms of type $Q_{\alpha,\beta}$ in two dimensions is already carried out by Zhou in \cite{Zhou}.  His proof is done using spaces $N^{s+1,\theta}$ with the norm given by\footnote{see \cite{SelbergT}  Section 3.5 for a comparison with $\Hs$ spaces.}
\be
\|u\|_{N^{s+1,\theta}}=\|\Lambda_+^{s+1}\Lambda^\theta_-u\|_{L^2}\label{xsth}.
\ee
In his work $\theta=s+\frac{1}{2}$.  We state Zhou's result. 
\begin{theorem*} (\cite{Zhou})
Consider in $\R^{2+1}$ the space-time norms \eqref{xsth} and functions $\varphi, \psi$ defined on $\R^{2+1}$, the estimates
\[
\|Q_{\alpha\beta}(\varphi,\psi)\|_{N^{s,s-\frac{1}{2}}}\lesssim \|\varphi\|_{N^{s+1,s+\frac{1}{2}}}\|\psi\|_{N^{s+1,s+\frac{1}{2}}}
\]
hold for any $\frac{1}{4} < s < \frac{1}{2}$.
\end{theorem*}
Our iteration is done using spaces $\Hs$.  Inspection of Zhou's proof shows that it could be easily modified to be placed in the context of $\Hs$ spaces.  However, even though our auxiliary functions' iterates $u_j$ and $v_j$ belong to $\Hs$, from \eqref{newphidf} we only have
\be
df \in H^{s,\theta} \Rightarrow \|\Lambda^s\Lambda^\theta_-Df\|_{L^2(\R^{2+1})} < \infty,\label{df}
\ee
but again inspection of Zhou's proof shows we can still handle $Q_{12}(f,f)$ given only that \eqref{df} holds.  Moreover, Zhou's proof works for $\frac{1}{4}<s<\frac{1}{2}$, but studying of his proof motivated an alternate proof that uses $\Hs$ and works for all values of $s>\frac{1}{4}$.  The proof is closely related to the original proof in \cite{Zhou}, but on the surface it seems more concise.  The reason for this is that we use Theorem F from \cite{KS}, which involves all the technicalities.  See \cite{Czubak} for the details. 
\end{subsubsection}
\begin{subsubsection}{Null Forms--Proof of Estimate (\ref{M3})}\label{nullformq}
We need
\[
\|[\partial_j f,\phi]\|_{H^{s,\theta-1+\epsilon}} \lesssim \|\partial_j f\|_\h\|\phi\|_\h, \quad j=1,2. 
\]
However analysis of the first iterate shows that for this estimate to hold we need $s > \frac{3}{4}$, so we need to work a little bit harder, and use \eqref{newphidf}\footnote{The obvious way is to just substitute for $\phi$ and leave $\partial_jf$ the same, but it is an exercise to see that this does not work (for several reasons!).}
\be\label{dfphibracket}
[\partial_j f, \phi]  =\frac{1}{4}[R_j (\partial_t u + iDu -\partial_t v + iDv),\partial_t u + iDu +\partial_t v - iDv].
  										\ee
If we use the bilinearity of the bracket, we can group \eqref{dfphibracket} by terms involving brackets of $u$ with itself, $v$ with itself, and then also by the terms that are mixed i.e., involve both $u$ and $v$.  So we have
\begin{align*}
4[\partial_j f, \phi]&=[R_j(\partial_t + iD)u,(\partial_t + iD)u]-[R_j(\partial_t - iD)v,(\partial_t - iD)v]\\
										&\quad +[R_j(\partial_t +iD)u,(\partial_t - iD)v]-[R_j(\partial_t - iD)v,(\partial_t + iD )u]. 
\end{align*}
Since $u$ and $v$ are matrix valued and do not commute we need to combine the last two brackets to take advantage of a null form structure.  This corresponds to \eqref{newnullformplus2} below (note the plus sign in the formula).\newline\indent
The needed estimates are contained in the following theorem
\begin{thm} Let $s>\frac{1}{4}$ and
\begin{align*}
&\frac{3}{4}-\frac{\epsilon}{2}<\theta\leq s+\frac{1}{2},\qand \theta< 1-\epsilon,\\
&0\leq \epsilon<\min\left( 2s-\frac{1}{2},\frac{1}{2}\right).
\end{align*}
and let $Q(\varphi,\psi)$ be given by
\begin{align}
Q(\varphi,\psi)&=(\partial_t \pm iD)R_j\varphi(\partial_t \pm iD)\psi - (\partial_t \pm iD)\varphi (\partial_t\pm iD)R_j\psi\\
\mbox{or}&\nonumber\\
Q(\varphi,\psi)&=(\partial_t \pm iD)R_j\varphi(\partial_t \mp iD)\psi + (\partial_t \pm iD)\varphi (\partial_t\mp iD)R_j\psi.\label{newnullformplus2}
\end{align}
Then
\be
Q(\Hs,\Hs) \hookrightarrow H^{s,\theta-1+\epsilon} 
\ee
or equivalently, the following estimate holds
\be\label{nullformest}
\|Q(\varphi,\psi)\|_{H^{s,\theta-1+\epsilon}} \lesssim \|\varphi\|_\Hs\|\psi\|_\Hs. 
\ee
\end{thm}
\begin{proof}
We show the details only for 
\[
(\partial_t + iD)R_j\varphi(\partial_t - iD)\psi + (\partial_t + iD)\varphi (\partial_t- iD)R_j\psi
\]
as the rest follows similarly.  Observe the symbol of $Q$ is
\[
q(\tau,\xi,\lambda,\eta)=\big(\frac{\xi_j}{|\xi|}+\frac{\eta_j}{|\eta|}\big)(\tau +|\xi|)(\lambda-|\eta|).
\]
Suppose $\tau \lambda \geq 0$, then
\[
q\leq 2\big|(\tau +|\xi|)(\lambda-|\eta|)\big|\leq \left\{\begin{array}{l} 2\big||\tau|+|\xi|\big|\big||\lambda|-|\eta|\big|\quad\mbox {if}\mbox\quad \tau, \lambda \geq 0,\\
																											 2\big||\tau|-|\xi|\big|\big||\eta| +|\lambda| \big|\quad\mbox {if}\mbox\quad \tau,\lambda\leq 0.	 \end{array}\right.
\]
It follows
\be\label{someestimate}
\iint_{\tau\lambda\geq 0}|\Lambda^s\Lambda_-^{\theta-1+\epsilon}Q(\varphi,\psi)|^2d\tau d\xi \lesssim \| D_+\varphi D_-\psi\|^2_{H^{s,\theta-1+\epsilon}}  +\|D_-\varphi D_+\psi\|^2_{H^{s,\theta-1+\epsilon}}  
\ee
and the estimate follows by Theorem \ref{dudv} below.\newline\indent
Suppose $\tau \lambda < 0$.  If we break down the computations into two regions
\be
\{ (\tau,\xi),(\lambda,\eta) : |\tau| \geq 2|\xi| \quad\mbox{or}\quad  |\lambda| \geq 2 |\eta|\} \quad\mbox {and} \quad \mbox {otherwise},\label{regions} 
\ee
then in the first region, we bound $q$ by
\[
q\leq 2(|\tau| +|\xi|)(|\lambda|+|\eta|)
\]
since there we do not need any special structure\footnote{It is a simple exercise in the first region.  See Appendix B in \cite{Czubak}.}.\newline\indent
In the second region, we have
\[
q\leq 4 |\xi||\eta|\left|\frac{\xi_i}{|\xi|}+\frac{\eta_i}{|\eta|}\right|,
\]
which is the absolute value of the symbol of the null form $Q_{tj}$ in the first iterate.  It has received a lot of attention, but we have not seen a reference, where it was discussed in the context other than of the initial data in $H^{s+1}\times H^{s}.$  This may be, because it has not come up as a nonlinearity before, and/or because it can be handled in the same way as the null form $Q_{ij}$.  The details are in \cite{Czubak}.
\end{proof}
Now we prove an estimate needed to show \eqref{someestimate} is bounded by the square of the right hand side of \eqref{nullformest}.
\begin{thm}\label{dudv}
Let $s>0$ and 
\begin{align*}
&\max\left(\frac{1}{2},1-s\right)<\theta<1,\\
&0\leq \epsilon\leq 1-\theta,
\end{align*}
then
\[
\|D_+\varphi D_-\psi\|_{H^{s,\theta-1+\epsilon}}\lesssim\|\varphi\|_\Hs\|\psi\|_\Hs
\]
\end{thm}
\begin{proof}
We would like to show
\[
\|\Lambda^s \Lambda^{\theta-1+\epsilon}_- (D_+\varphi D_-\psi)\|_{L^2(\R^{2+1})} 
\lesssim\|\varphi\|_{\mathcal H^{s+1,\theta}}\|\psi\|_{\mathcal H^{s+1,\theta}}.
\]
This follows from showing
\[
H^{s,\theta} \cdot \mathcal H^{s+1,\theta-1}  \hookrightarrow H^{s,\theta-1+\epsilon},
\]
which by the product rule\footnote{On $L^2$ this is very easy to establish using triangle inequality.  See \cite{KS}.} for the operator $\Lambda^s$ in turn follows from
\begin{align*}
H^{0,\theta} \cdot \mathcal H^{s+1,\theta-1}   \hookrightarrow H^{0,\theta-1+\epsilon},\\
H^{s,\theta} \cdot \mathcal H^{1,\theta-1}  \hookrightarrow H^{0,\theta-1+\epsilon}.
\end{align*}
It is easy to check
\[
\mathcal H^{s+1,\theta-1}\hookrightarrow H^{s+1+\theta-1,0} \qand \mathcal H^{1,\theta-1}  \hookrightarrow H^{\theta,0},
\]
so we just need to show
\begin{align*}
H^{0,\theta}\cdot H^{s+\theta,0}\hookrightarrow H^{0,\theta-1+\epsilon},\\
H^{s,\theta} \cdot H^{\theta,0}\hookrightarrow H^{0,\theta-1+\epsilon},
\end{align*}
which are weaker than
\begin{align*}
H^{0,\theta}\cdot H^{s+\theta,0}\hookrightarrow L^2,\\
H^{s,\theta} \cdot H^{\theta,0}\hookrightarrow L^2,
\end{align*}
but those follow from the Klainerman-Selberg estimate \eqref{E} as long as $s+\theta > 1$, which holds by the conditions we impose on $s$ and $\theta$.
\end{proof}
An alternate approach could be to follow the set up used by \cite{KlainermanMachedon95} and estimate the integral directly.
\end{subsubsection}

\begin{subsubsection}{Elliptic Piece: Proof of Estimate (\ref{M2})}\label{ellpiece}
Recall we wish to show
\be
\|A_0w\|_{H^{s,\theta-1+\epsilon}} \lesssim \|A_{0}\| \|w\|_\h.\label{M21}
\ee
We need this estimate during our iteration, so we really mean $A_{0,j}$, but for simplicity we omit writing of the index $j$.  Now we choose a norm for $A_0$ to be anything that makes \eqref{M21} possible to establish.   This results in 
\[
\|A_0\|=\|A_0\|_{L^{\p}_tL^\infty_x}+\|D^sA_0\|_{L^{p}_tL^q_x},
\]
where
\be\label{aspqs}
\bs
&\p \in \left(1-2s,\frac{1}{2}\right),\\
\frac{2}{p}=1-\frac{1}{q}, \quad  &\max\left(\frac{1}{3}(1-2s),\frac{s}{2}\right) < \frac{1}{q} < \frac{2}{3}s.
\end{split}
\ee
For now we assume we can show $A_0 \in L^{\p}_tL^\infty_x \cap L^{p}_t\dot W^{s,q}_x$ and delay the proof to Section \ref{ellipticpiece1}, where the reasons for our choices of $\p, p, q$ should become clear.
We start by using $\theta-1+\epsilon < 0$ 
\be\label{beginning}
\|A_{0}w\|_{H^{s,\theta-1+\epsilon}} \leq \|\Lambda^s (A_{0}w) \|_{L^2(\R^{2+1})} \lesssim \|A_{0}w \|_{L^2(\R^{2+1})} + \|D^s(A_{0}w) \|_{L^2(\R^{2+1})} 
\ee
For the first term by H\"older's inequality
\be\label{firstT}
\begin{split}
\|A_{0}w \|_{L^2(\R^{2+1})}&\leq \|A_0\|_{L^\p_tL^\infty_x}\|w\|_{L^{\p'}_tL^{2}_x},\quad \frac{1}{\p}+\frac{1}{\p'}=\frac{1}{2}, \p \mbox{ as in }\eqref{aspqs} \\
&\lesssim\|A_0\|\|w\|_{H^{0,\theta}},\quad\mbox{by \eqref{C}}\\
&\leq \|A_0\|\|w\|_{H^{s,\theta}}.
\end{split} 
\ee
We bound the second term in \eqref{beginning}  by
\[
\|D^s(A_{0}w) \|_{L^2(\R^{2+1})}\lesssim \underbrace{\|A_0\|_{L^\p_tL^\infty_x}\|D^s w\|_{L^{\p'}_tL^{2}_x}}_{I}+
\underbrace{\|D^s A_0\|_{L^{p}_tL^q_x}\|w\|_{L^{p'}_tL^{q'}_x}}_{II}
\]
where $\frac{1}{p}+\frac{1}{p'}=\frac{1}{2}=\frac{1}{q}+\frac{1}{q'}$ and $p,q$ are as in\eqref{aspqs} and $\p$ as in \eqref{firstT}.   $I$ is handled similarly to \eqref{firstT} as follows.  Apply \eqref{C} with $u=D^sw$ to obtain\footnote{
Or we could bound $\|D^sw\|_{L^{\p'}_tL^{2}_x}$ by $\|\Lambda^sw\|_{L^{\p'}_tL^{2}_x}$ and apply \eqref{C} with $u=\Lambda^s u$.}
\be\label{usingC}
I\lesssim \|A_0\|\|D^s w\|_{H^{0,\theta}}\leq \|A_0\|\|w\|_{H^{s,\theta}}.
\ee 
We now consider II.  By the choices of $p,q,$ Klainerman-Selberg estimate \eqref{D} applies\footnote{See the discussion in Section \ref{Aestneeded} for an explanation.} and gives
\be\label{usingD2}
II \leq \|A_0\|\|w\|_{L^{p'}_tL^{q'}_x}\lesssim \|A_0\|\|w\|_{H^{1-\frac{2}{q'}-\frac{1}{p'},\theta}}.
\ee
From \eqref{aspqs} we also have
\be\label{usingD3}
II \lesssim\|A_0\|\|w\|_{H^{1-\frac{2}{q'}-\frac{1}{p'},\theta}}\lesssim \|A_0\|\|w\|_{H^{s,\theta}}.
\ee
and \eqref{M21} follows now from \eqref{firstT}, \eqref{usingC} and \eqref{usingD3}.
\begin{remark}\label{estinMKG}
The above proof illustrates other difficulties due to working in $2$ dimensions.  Initially, we wanted to follow Selberg's proof of estimate (38) in \cite{Selberg}, and just use $\|\Lambda^s A_0\|_{L^p_tL^q_x}$ norm.  Unfortunately in $2D$, the condition $sq>2$ needed to show $A_0 \in L^p_tL^\infty_x$ is disjoint from conditions needed to use Klainerman-Tataru estimate \eqref{A} and establish that $\Lambda^s A_0 \in {L^p_tL^q_x}$ in the first place.  This resulted in the $L^{\p}_tL^\infty_x \cap L^{p}_t\dot W^{s,q}_x$ space above and also having to employ Klainerman-Selberg estimate \eqref{D}, which was not needed in \cite{Selberg} for the proof of (38).
\end{remark}
\end{subsubsection}
\end{subsection}


\begin{subsection}{Elliptic Regularity: Estimates for $A_0$.}\label{ellipticpiece1}
Here we present a variety of a priori estimates for the nondynamical variable $A_0$.  At each point we could add the index $j$ to $A_0, df$ and $\phi$.  Therefore the presentation also applies to the iterates $A_{0,j}$.  It is an exercise to show that the estimates we obtain here are enough to solve for $A_{0,j}$ at each step as well as to close the iteration for $A_0$. 
Let $A_0$ solve 
\begin{align*}
\triangle A_0 =d^\ast [A_0,\ast df]+ d^\ast[df, \phi ]=-\partial_1[A_0,\partial_2 f]+\partial_2[A_0,\partial_1 f]+\partial_i[\partial_if,\phi].
\end{align*}
There is a wide range of estimates $A_0$ satisfies.  Nevertheless, the two spatial dimensions limit our ``range of motion.''  For example, it does not seem possible to place $A_0(t)$ in $L^2$.  We state the general results and only show the cases we need to prove $A_0 \in L^\p_tL^\infty_x \cap L^p_t\dot W^{s,q}_x$ as required in the last section.  The rest of the cases can be found in \cite{Czubak}.  We add that the proofs of both of the following theorems were originally inspired by Selberg's proof of his estimate (45) in \cite{Selberg}.  We start with the homogeneous estimates.
\begin{thm}\label{thm1}  Let $s > 0$, and let $0\leq a\leq s+1$ be given.  And suppose $1\leq p\leq \infty$ and $1<q<\infty$  satisfy 
\begin{align}
\max\left(\frac{1}{3}(1+2a-4s), \frac{1}{2}(1+a-4s), \frac{1}{2}\min(a,1)\right) < \frac{1}{q} < \frac{1+a}{2},\label{cq1}\\
1-\frac{2}{q}+a-2s \leq \frac{1}{p} \leq \frac{1}{2}\left(1-\frac{1}{q}\right),
\quad \frac{1}{p}<\left(1-\frac{2}{q}+a\right)\label{cp1}.
\end{align}
\begin{itemize}
\item[i)] If $0\leq a\leq 1$ and the $\h$ norm of $\nabla f$ is sufficiently small, then  $A_0 \in L^p_t\dot W^{a,q}_x$ and we have the following estimate
\be
\|A_0\|_{L^p_t\dot W^{a,q}_x}\lesssim \|\phi\|_\h\|\nabla f\|_\h.
\ee
\item[ii)] If $1<a\leq s+1$ and $A_0 \in L^p_tL^{(1/q-1/2)^{-1}}_x$, then $A_0 \in  L^p_t\dot W^{a,q}_x$ and we have the following estimate
\be
\|A_0\|_{L^p_t\dot W^{a,q}_x}\lesssim (\|A_0\|_{L^p_tL^{(1/q-1/2)^{-1}}_x}+\|\phi\|_\h)\|\nabla f\|_\h.
\ee
\end{itemize}
\end{thm}
\begin{cor}\label{corr1}
Let $s>0$, then $A_0 \in C_b(I:\dot H^a_x)$, where
\[ 0 < a \leq
   \left\{ \begin{array}{l} 
\begin{split}
 2s&\quad\mbox{if}\quad 0<s\leq 1\\
 1+s&\quad\mbox{if}\quad 1< s
\end{split}
\end{array} \right. \]
\end{cor}
\begin{proof}[Proof of Corollary \ref{corr1}]
Suppose $0<s<\frac{1}{2}$.  Then use part i) of the theorem with $q=2$ and $p=\infty$ to obtain $A_0 \in L^\infty_t\dot H^a_x$ for $a\leq 2s$.  $A_0$ continuous as a function of time easily follows from a contraction argument in $C_b(I:\dot H^a_x)$ using $L^\infty_t\dot H^a_x$ estimates.  $s\geq\frac{1}{2}$ is considered in \cite{Czubak}.
\end{proof}
So far we just need $s>0$ in order to make the estimates work.  The requirement for $s>\frac{1}{4}$ does not come in till we start looking at the nonhomogeneous spaces, where also the range of $p$ and $q$ is smaller.  However, we can distinguish two cases $aq<2$ and $aq > 2$.
\begin{thm}\label{thm2}  Let $s > 0$, and suppose the $\h$ norm of $\nabla f$ is sufficiently small.   
\begin{itemize}
\item[i)]  If $aq < 2$ for $0< a <(2s,1)$ and if $p$ and $q$ satisfy 
\begin{align}
\max\left(\frac{1}{2}+a-2s, \frac{a}{2}\right) < &\frac{1}{q} < \frac{1}{2},\label{cq2}\\
1-\frac{2}{q}+a-2s \leq &\frac{1}{p} < \frac{1}{2} -\frac{1}{q},\label{cp2}
\end{align}
then $A_0 \in L^p_tW^{a,q}_x$ and we have the following estimate
\be
\|A_0\|_{L^p_tW^{a,q}_x}\lesssim \|\phi\|_\h\|\nabla f\|_\h.
\ee
\item[ii)] If $aq > 2$, then we need $s>\frac{1}{4}$ and $0<a<\min(4s-1,1+s,2s)$.  Suppose $p$ and $q$ also satisfy
\begin{align}
\max\left(\frac{a-s}{2},\frac{1}{2}+a-2s\right) &< \frac{1}{q} < \frac{1}{2}\min(a,1),\label{cq3}\\
1-\frac{2}{q}+a-2s &\leq \frac{1}{p} < \frac{1}{2} -\frac{1}{q},\label{cp3}
\end{align}
then $A_0 \in L^p_tW^{a,q}_x$ and we have the following estimate
\be
\|A_0\|_{L^p_tW^{a,q}_x}\lesssim \|\phi\|_\h\|\nabla f\|_\h.
\ee
\end{itemize}
\end{thm}
\begin{cor}\label{corr2}
If $s>\frac{1}{4}$ and the $\h$ norm of $\nabla f$ is sufficiently small, we have in particular $A_0 \in L^p_tL^\infty_x$ for $p$ satisfying
\be
1-2s < \frac{1}{p}< \frac{1}{2},
\ee
and we have the following estimate
\be
\|A_0\|_{L^p_tL^\infty_x}\lesssim \|\phi\|_\h\|\nabla f\|_\h.
\ee
\end{cor}
\begin{proof}[Proof of Corollary \ref{corr2}]
For each $p \in (1-2s,\frac{1}{2})$ we can find some $a$ and $q$, which satisfy the conditions of Theorem \ref{thm2}, part ii).  The corollary then follows from the Sobolev embedding: $W^{a,q}(\mathbb R^2) \hookrightarrow L^\infty(\mathbb R^2)$ for $aq > 2$.
\end{proof}
\begin{remark}
Here we also would like to emphasize the arrival of the necessity of $s>\frac{1}{4}$.  Conditions on $\frac{1}{p}$ in
\eqref{cp3} are needed so we can use below the Klainerman-Tataru estimate \eqref{A}.  In order to be able to choose such  $\frac{1}{p}$, obviously $1-\frac{2}{q}+a-2s$ must be strictly less than $\frac{1}{2} -\frac{1}{q}$.  This forces $\frac{1}{q}$ to be strictly greater than $\frac{1}{2}+a-2s$. We also need $aq>2$ to use the Sobolev embedding in Corollary \ref{corr2}, so if we want to be able to find $q$ between $\frac{1}{2}+a-2s$
and $\frac{a}{2}$, a is forced to be strictly less than $4s-1$.  Therefore $s$ must be greater than $\frac{1}{4}$.  See below for another instance of requiring $s>\frac{1}{4}$.
\end{remark}

\subsection{Proof of estimates needed in \ref{ellpiece}}\label{Aestneeded}
Recall we would like to show $A_0 \in L^{\p}_tL^\infty_x \cap L^p_t\dot W^{s,q}_x$.
Therefore, we are interested in part i) of Theorem \ref{thm1} and part ii) in Theorem \ref{thm2}, so we can conclude Corollary \ref{corr2}.  Moreover, we need a specific case of part i) in Theorem \ref{thm1}, because we need $A_0 \in L^p_t\dot W^{s,q}_x$, where $p, q$ in addition satisfy
\be\label{newpq}
 1-\frac{2}{p} \leq \frac{1}{q}<\frac{1}{2}, \qand \frac{2}{q}-\frac{1}{2}+\frac{1}{p}\leq s,
\ee
so we can use
\be
H^{s,\theta} \hookrightarrow H^{1-(1-\frac{2}{q})-(\frac{1}{2}-\frac{1}{p}),\theta}(\R^{2+1}) \hookrightarrow L^{(1/2-1/p)^{-1}}_tL^{(1/2-1/q)^{-1}}_x,
\ee
in \eqref{usingD2} and \eqref{usingD3}.  When we put \eqref{newpq} together with \eqref{cq1} and \eqref{cp1} with $a=s$, we obtain second line of \eqref{aspqs}, namely
\be\label{aspqs1}
\frac{2}{p}=1-\frac{1}{q}, \quad  \max\left(\frac{1}{3}(1-2s),\frac{s}{2}\right) < \frac{1}{q} <  \frac{2}{3}s.
 \ee
\begin{remark}
Observe that in order to be able to find such $q$ we must have $s>\frac{1}{4}$.
\end{remark}
Consider
\be\label{e3}
\begin{split}
\|A_0\|_{L^p_t\dot W^{s,q}_x}&=\|\triangle^{-1}(d^\ast [A_0,\ast df]+ d^\ast[df, \phi ])\|_{L^p_t\dot W^{s,q}_x}\\
 														 &\lesssim\|D^{-1}(A_0\nabla f)\|_{L^p_t\dot W^{s,q}_x}+ \|D^{-1}(\nabla f\phi)\|_{L^p_t\dot W^{s,q}_x}\\
 														 &\lesssim\|D^{s-1}(A_0\nabla f)\|_{L^p_tL^q_x}+ \|D^{s-1}(\nabla f\phi)\|_{L^p_tL^q_x}\\
 														 &\lesssim\|A_0\nabla f\|_{L^p_tL^r_x}+ \|D^{s-1}(\nabla f\phi)\|_{L^p_tL^q_x},
\end{split}
\ee
where we use the Sobolev embedding with $\frac{1}{q}=\frac{1}{r}-\frac{1-s}{2}.$  The latter term is bounded
by $\|\nabla f\|_\h\|\phi\|_\h$ using the Klainerman-Tataru estimate \eqref{A}, whose application we discuss in the section below.  For the former we use $\frac{1}{r}=\frac{1}{q}+\frac{1-s}{2}=(\frac{1}{q}-\frac{s}{2})+\frac{1}{2}$
\be
\|A_0\nabla f\|_{L^p_tL^r_x} \leq \|A_0\|_{L^p_tL_x^{(1/q-s/2)^{-1}}}\|\nabla f\|_{L^\infty_tL^2_x}\lesssim \|A_0\|_{L^p_t\dot W^{s,q}_x}\|\nabla f\|_{\h}.
\ee
Then if the $\h$ norm of $\nabla f$ is sufficiently small, we obtain
\be
\|A_0\|_{L^p_t\dot W^{s,q}_x}\lesssim \|\nabla f\|_\h\|\phi\|_\h,
\ee
as needed.
\newline\indent
For the non-homogeneous estimate, since here $\frac{1}{4}<s<\frac{1}{2}$ the upper bound for $a$ is simply $4s-1$.  In addition, for our purposes right now it suffices to show the estimate for one particular $a$.  Therefore we set $0<a<\min(s,4s -1)$ for $\frac{1}{4}<s<\frac{1}{2}$, and we let $p,q$ satisfy \eqref{cq3} and \eqref{cp3}.   We have
\be\label{e4}
\begin{split}
\|A_0\|_{L^p_t W^{a,q}_x}&\lesssim\|D^{-1}(A_0\nabla f)\|_{L^p_t W^{a,q}_x}+ \|D^{-1}(\nabla f\phi)\|_{L^p_t W^{a,q}_x}\\
 												 &\lesssim\|D^{-1}(A_0\nabla f)\|_{L^p_tL^q_x}+ \|D^{-1}(\nabla f\phi)\|_{L^p_tL^q_x}\\
 												 &\quad +\|D^{a-1}(A_0\nabla f)\|_{L^p_tL^q_x}+\|D^{a-1}(\nabla f\phi)\|_{L^p_tL^q_x}.
\end{split}
\ee
Klainerman-Tataru estimate \eqref{A} handles the second and the last term (see below).  Consider the first term
\be
\bs
\|D^{-1}(A_0\nabla f)\|_{L^p_tL^q_x} &\lesssim\|A_0\nabla f\|_{L^p_tL^r_x},\quad \frac{1}{q}=\frac{1}{r}-\frac{1}{2},\\
															 &\leq \|A_0\|_{L^p_tL^q_x}\|\nabla f\|_{L^\infty_tL^2_x}\\
	 &\leq \|A_0\|_{L^p_tW^{a,q}_x}\|\nabla f\|_{\h}.
\end{split}
\ee
For the third term we have
\be\nonumber
\begin{split}
\|D^{a-1}(A_0\nabla f)\|_{L^p_tL^q_x}&\lesssim \|A_0\nabla f\|_{L^p_tL^r_x},\quad \frac{1}{q}=\frac{1}{r}-\frac{1-a}{2}\\
															 &\lesssim \|A_0\|_{L^p_tL^q_x}\|D^a\nabla f\|_{L^\infty_tL^2_x}, \quad \frac{1}{r}=\frac{1}{q}+(\frac{1}{2} -\frac{a}{2})\\
											         &\lesssim \|A_0\|_{L^p_tW^{a,q}_x}\|\nabla f\|_{\h},
\end{split}
\ee
Then as before, this completes the proof if the $\h$ norm of $\nabla f$ is sufficiently small. 
\end{subsection}
\end{section}

\subsection{Applying Klainerman-Tataru Theorem}\label{KTT}\nonumber
We said that several estimates above follow from the Klainerman-Tataru estimate \eqref{A}.  We need to check that this is in fact the case.  We begin by stating the theorem.  We state it for two dimensions only, and as it was given in \cite{KS} (the original result holds for $n\geq 2$).
\begin{theorem*}(\cite{KlainermanTataru})
Let $1 \leq p \leq \infty$, $ 1 \leq q < \infty$.  Assume that
\begin{eqnarray}
 \frac{1}{p} \leq \frac{1}{2}\left(1-\frac{1}{q}\right),\label{c1}\\
0 < \sigma < 2\left(1-\frac{1}{q} -\frac{1}{p}\right),\label{c2}\\
s_1, s_2 < 1 - \frac{1}{q}-\frac{1}{2p},\label{c3}\\
s_1+s_2+\sigma=2(1- \frac{1}{q}-\frac{1}{2p}).\label{c4} 
\end{eqnarray}
Then 
\[
\|D^{-\sigma}(uv)\|_{L^p_tL^q_x(\R^2)} \lesssim \|u\|_{H^{s_1,\theta}} \|v\|_{H^{s_2,\theta}},
\]
provided $\theta > \frac{1}{2}$.
\end{theorem*}
The first time we use the theorem is in \eqref{e3} for the term $\|D^{s-1}(\nabla f\phi)\|_{L^p_tL^q_x}$.  Note $\sigma=1-s$.  Clearly $1 \leq p \leq \infty$, $ 1 \leq q < \infty$.  Next by \eqref{aspqs1} 
$\frac{2}{p}=1-\frac{1}{q}$, so \eqref{c1} holds.  Since $s<\frac{1}{2}$, $\sigma>0$, and we can see \eqref{c2} holds when we substitute $\frac{1}{2}-\frac{1}{2q}$ for $\frac{1}{p}$ in the right hand side and use $\frac{1}{q}<\frac{2}{3}s$.  Next we let $s_1=s_2$ and with $\sigma=1-s>0$, \eqref{c4} implies \eqref{c3}, so we only check \eqref{c4}.  To that end we must be able to choose $s_1$ so that $2s_1=1-\frac{2}{q}-\frac{1}{p}+s\leq 2s,$
which is equivalent to our condition on $p$ and one of the lower bounds on $\frac{1}{q}$.\\
\indent The next place we use the theorem is in \eqref{e4} for $\|D^{-1}(\nabla f\phi)\|_{L^p_tL^q_x}$, $\|D^{a-1}(\nabla f\phi)\|_{L^p_tL^q_x}$, where $p$ and $q$ are as in \eqref{cq3} and \eqref{cp3} with $0<a<\min(s,4s -1)<1$.  Then for $\sigma=1$, by the right hand side of \eqref{cp3}, \eqref{c2} holds and implies \eqref{c1}.  Note, since \eqref{c2} is true with $\sigma=1$, it is true with $\sigma=1-a$.  Next, for $\sigma=1$ \eqref{c4} gives \eqref{c3} and also for $\sigma=1-a$ as long as $0<a<1$.  So again it is sufficient to see we can have $s_1$ defined by \eqref{c4} such that $s_1\leq s$, but for $\sigma=1-a$ that follows from the left hand side of \eqref{cp3}, and shows we can find it for $\sigma=1$ as well.
\bibliography{magdasbib}
\bibliographystyle{plain}
\vspace{.125in}
 
\end{document}